\documentclass[12pt,thmsa, reqno]{amsart}

\usepackage[dvips]{graphics}
\input{amssym.def}
\input{amssym.tex}

\def \x{{\bf x}}
\def \s{{\bf s}}
\def \r{{\bf r}}
\def \v{{\bf v}}
\def \y{{\bf y}}
\def \u{{\bf u}}

\def\gk{\Gamma_k}

\def\Cal{\mathcal}

\def\pk{{\Cal P}_k}

\def\pq{{\Cal P}_q}

\def\A{{\Cal A}}

\def\P{{\Cal P}}

\def\S{{\Cal S}}

\def\gk{\Gamma_k}

\def\pk{\P_k}

\def\bbr{{\Bbb R}}

\def\rank{{\hbox{\rm rank}}}

\def\tr{{\hbox{\rm tr}}}

\def\det{{\hbox{\rm det}}}

\def\gk{\Gamma_k}

\def\pk{\P_k}

\def\rn{\bbr^n}
\def\sn{S^{n-1}}

\def\part{\partial}
\def\intl{\int\limits}
\def\b{\beta}

\def\Gam{\Gamma}

\def\a{\alpha}

\def\Del{\Delta}

\def\gam{\gamma}

\def\sig{\sigma}

\def\z{\zeta}
\def\th{\theta}

\def\t{\tau}

\def\chi{{\bf 1}}

\font\frak=eufm10

\def\fr#1{\hbox{\frak #1}}

\def\frM{\fr{M}}

\newtheorem{theorem}{Theorem}[section]
\newtheorem{lemma}[theorem]{Lemma}

\theoremstyle{definition}

\theoremstyle{remark}
\newtheorem{remark}[theorem]{Remark}

\theoremstyle{corollary}
\newtheorem{corollary}[theorem]{Corollary}

\numberwithin{equation}{section}

\newcommand{\be}{\begin{equation}}
\newcommand{\ee}{\end{equation}}

\newcommand{\bea}{\begin{eqnarray}}
\newcommand{\eea}{\end{eqnarray}}
\newcommand{\Bea}{\begin{eqnarray*}}
\newcommand{\Eea}{\end{eqnarray*}}

\def\sideremark#1{\ifvmode\leavevmode\fi\vadjust{\vbox to0pt{\vss
 \hbox to 0pt{\hskip\hsize\hskip1em
\vbox{\hsize2cm\tiny\raggedright\pretolerance10000
 \noindent #1\hfill}\hss}\vbox to8pt{\vfil}\vss}}}%

\begin{document}

\title[Blaschke-Petkantschin Formula]
{On the Blaschke-Petkantschin Formula and Drury's Identity}

\author{B. Rubin}
\address{
Department of Mathematics, Louisiana State University, Baton Rouge,
LA, 70803 USA}

\email{borisr@lsu.edu}





\begin{abstract}
The Blaschke-Petkantschin formula is a variant of the  polar decomposition of the $k$-fold Lebesgue measure on $\mathbb {R}^n$ in terms of the corresponding  measures on $k$-dimensional linear subspaces of $\mathbb {R}^n$. We suggest a new elementary proof of this formula  and discuss its connection with the celebrated Drury's identity that plays a key role in the study of mapping properties of the Radon-John $k$-plane transforms.  We give a new derivation of this identity and provide it with precise information about constant factors and the class of admissible functions.

\end{abstract}

\maketitle

\section{Introduction}

\setcounter{equation}{0}

The classical  Blaschke-Petkantschin formula gives  decomposition of the Euclidean measure on $k$ copies of $\rn$ into the corresponding measures on $k$-dimensional subspaces of $\rn$ with the relevant Jacobian; see the case $q=k$ in (\ref{BlP}) below. After the
 pioneering works by Blaschke \cite{Bla} and  Petkantschin \cite{P}, this formula and its modifications arise in different aspects
 of Analysis, Integral Geometry, Multivariate Statistics, and Probability; see, e.g., \cite{BL, DPP, J, M, Mi, San, SchW92, SchW}, to mention a few.  The books \cite{SchW92, SchW} contain a nice history of the subject.

The simplest case of the Blaschke-Petkantschin formula corresponds to  $k=1$ when the standard polar decomposition yields
 \be \label{f1} \intl_{\rn} \!f(x)\, dx=\!\!\intl_{\sn}\! \!d\th \intl_0^\infty \!f(r\th) r^{n-1}\, dr\!=\!
\frac{1}{2}\intl_{\sn} \!\!\!d\th \!\intl_{-\infty}^\infty \!\!f(r\th) |r|^{n-1}\, dr.\ee
This gives
\be \label{f2} \intl_{\rn} f(x)\, dx=\frac{\sig_{n-1}}{2}\intl_{G_{n,1}}\!\! d\ell \intl_{\ell} \!f(x) |x|^{n-1} d_\ell x, \ee
where $G_{n,1}$ is the Grassmann manifold of lines $\ell$ through the origin, $d\ell$ is the standard probability measure on $G_{n,1}$, $\sig_{n-1}$ denotes the area of the unit sphere $\sn$ in $\rn$, and $d_\ell x$ stands for the Lebesgue measure on $\ell$.

The present article is influenced by   intimate connection between the   Blaschke-Petkantschin formula and its modification, known as {\it ``Drury's identity"}, which was independently discovered by Drury \cite{Dru84} in his study of norm estimates for the $k$-plane Radon-John transform;  see the case $k+\ell+1=n$ in (\ref{Buu1}). This remarkable connection was pointed out by  Baernstein II and  Loss \cite{BL}.
A number of breakthrough applications  of  Drury's identity to difficult problems related to $L^p$-$L^q$ estimates for $k$-plane transforms can be found in the works by Christ \cite{Chr84}, Drouot \cite{Dro14}, Flock \cite{Fl16},  Bennett, Bez,  Flock,   Gutiérrez, and  Iliopoulou \cite {BBFGI}.

Another motivation for  writing this article was Drury's observation \cite{Dru84} that his formula has a lot in common with
 analytic continuation of Riesz distributions on matrix spaces. These distributions and the corresponding potential operators were studied by  Stein \cite{St}, Gelbart \cite{Ge}, Khekalo \cite {Kh}, Ra\"{\i}s \cite{Ra}, Rubin \cite{Ru06a}, where one can find further  references.

Most of known proofs of the  Blaschke-Petkantschin formula employ either a super-powerful machinery  of differential forms (see, e.g., \cite {San, M, Mi}) or a clever inductive argument. In Section 2 we present an elementary proof of a slightly more general version of this formula, following the same idea as  in (\ref{f1})-(\ref{f2})  and using  the well-known polar decomposition of matrices. As a consequence, in Section 3  we formulate the affine version of this formula previously known under more restrictive assumptions, and obtain a slight generalization of Drury's identity  with sharp constant and precise information about  the class of admissible functions. This constant was not specified in \cite{Dru84, BL}. It was later  obtained in \cite {BBFGI} in a conceptually more complicated way than we do.

After finishing the first version of the paper, the author became aware of close works by Moghadasi \cite{Mog} and
Forrester \cite{Fo} devoted to application of the matrix polar decomposition to derivation of the Blaschke-Petkantschin formula. Our reasoning essentially differs from \cite{Fo, Mog}.

\section{Derivation of the Blaschke-Petkantschin Formula}

The  reasoning in (\ref{f1})-(\ref{f2}) extends to  functions $F(\x)=F(x_1, \ldots, x_k)$ on  $(\rn)^k$ if the latter is treated as the space $\frM_{n,k}$  of real matrices having $n$ rows and $k$  columns and the polar decomposition $dx=r^{n-1}\, dr \,d\th$  is replaced by its analogue for $d\x$. Below we recall  basic facts.

 If $\x=(x_1, \ldots, x_k)=(x_{i,j})\in \frM_{n,k}$, then
$d\x=\prod^{n}_{i=1}\prod^{k}_{j=1}  dx_{i,j}$ is the elementary volume in $\frM_{n,k}$.
 In the following
  $\x^T$ denotes the transpose of  $\x$ and $I_k$ is the identity $k \times k$
  matrix. Given a square matrix $\bf a$,  we denote by $|\bf a|$ the absolute value of the determinant of $\bf a$;  $\tr (\bf a)$
  stands for  the trace  of $\bf a$.

 Let $\S_k$ be the space of $k \times k$ real symmetric matrices
$\s=(s_{i,j}),
 \, s_{i,j}=s_{j,i}$. It is a measure space   isomorphic to $\bbr^{k(k+1)/2}$
 with the volume element $d\s=\prod_{i \le j} ds_{i,j}$.
We denote by  $\pk$ the  cone of positive definite
matrices in $\S_k$. For $n\geq k$, let $V_{n,k}= \{\v \in \frM_{n,k}: \v^T\v=I_k \}$
 be the Stiefel manifold of orthonormal $k$-frames in $\bbr^n$.
 If $n=k$, then $V_{n,n}=O(n)$ is the orthogonal group in $\bbr^n$.
  The group $O(n)$
 acts on $V_{n,k}$ transitively by the rule $g: \v\to g\v, \quad g\in
 O(n)$, in the sense of matrix multiplication.   We fix the corresponding
invariant measure $d\v$ on  $V_{n,k}$   normalized by
\be\label{2.16} \sigma_{n,k}
 \equiv  \! \intl_{V_{n,k}} d\v \! = \frac {2^k \pi^{nk/2}} {\gk
 (n/2)}, \ee
 where
\be\label{2.4}
 \gk (\a)\! =\! \intl_{\pk} \! \exp(-\tr (\r)) |\r|^{\a -(k+1)/2} d\r=\pi^{k(k-1)/4}\prod\limits_{j=0}^{k-1} \Gam (\a\! - \! j/2) \ee
is the  Siegel gamma  function associated to the cone $\pk$
 \cite[ p. 62]{Mu}. This integral  converges absolutely
 if and only if $Re \, \a>(k-1)/2$.

\begin{lemma}\label{l2.3} {\rm (polar decomposition).}
Let $\;\x \in \frM_{n,k}, \; n \ge k$. If  $\;\rank (\x)= k$, then $\x$ is uniquely decomposed as
\[
\x=\v\r^{1/2}, \qquad \v \in V_{n,k},   \qquad \r=\x^T\x \in\pk,\] and for $F\in L^1 (\frM_{n,k})$ we have
\be\label{2.44}
\intl_{\frM_{n,k}} F(\x)\, d\x=2^{-k}\intl_{V_{n,k}}d\v \intl_{\pk}  F(\v\r^{1/2})\, |\r|^{(n-k-1)/2} d\r. \ee
\end{lemma}

This statement and its generalizations can be found in different
sources, see, e.g., \cite[ p. 482]{Herz}, \cite[pp. 66, 591]{Mu},
\cite[p. 130]{FT}, \cite[Lemma 3.1]{Zha1}. For $\;\x \in \frM_{n,k}$, we set
\be
 |\x|_k =(\det (\x^T\x))^{1/2},\ee
 which  is the volume of the parallelepiped spanned  by the column-vectors $x_1, \ldots, x_k$ of the matrix $\x$
\cite[ p. 251]{G}.

 Our aim is to replace integration over the Stiefel manifold in (\ref{2.44}) by integration over the Grassmann manifold  $G_{n,k}$  of $k$-dimensional linear subspaces of $\rn$. We perform this replacement in a slightly more general fashion.  Given $\xi\in G_{n,k}$, we denote by $d\xi $ the $O(n)$-invariant probability measure on $G_{n,k}$ and write $\xi^q$ for the collection of $q$ copies of $\xi$.

\begin{theorem}\label{the} Let $1\le q\le k\le n$. If $F\in L^1 (\frM_{n,q})$,  then
\be\label{BlP}
\intl_{\frM_{n,q}} F(\x)\, d\x= \frac{\sigma_{n,q}}{\sigma_{k,q}} \intl_{G_{n,k}} d\xi \intl_{\xi^q} F(\x)\, |\x|_q^{n-k}\, d_\xi \x,\ee
where   $\x=(x_1, \ldots, x_q)$ and $d_\xi \x = d_\xi x_1 \ldots  d_\xi x_q $
 stands for the usual Lebesgue integration over $\xi^q$.
\end{theorem}
\begin{proof} The case $q=k$ in (\ref{BlP}) is the classical Blaschke-Petkantschin formula. To prove the theorem, we denote by $I$ the left-hand side of (\ref{BlP}).  By Lemma \ref{l2.3},
\bea I&=&2^{-q} \intl_{V_{n,q}} d\v \intl_{\pq} F(\v\r^{1/2})\,|\r|^{(n-q-1)/2} d\r \nonumber\\
&=&2^{-q}\sigma_{n,q} \intl_{O(n)} d\a   \intl_{\pq}  F \left (\a \left[\begin{array} {c} I_q \\ 0
\end{array} \right] \r^{1/2} \right )\,   |\r|^{(n-q-1)/2} d\r. \nonumber\eea
We replace $\a$ by
\[ \a  \left [\begin{array}{cc}  \b & 0 \\
0& I_{n-k}
\end{array}\right], \qquad \b \in O(k),\]
and integrate in $\b$ with respect to the standard probability measure on $O(k)$. Changing the order of integration, we get
\bea I&=&2^{-q} \sigma_{n,q} \intl_{O(n)} d\a \intl_{O(k)} d\b  \intl_{\pq}
 F \left (\a \left[\begin{array} {c} \b \left[\begin{array} {c} I_q \\ 0 \end{array} \right] \\ 0 \end{array} \right] \r^{1/2} \right )\,   |\r|^{(k-q-1)/2}\, |\r|^{(n-k)/2} d\r \nonumber\\
&=&\frac{2^{-q} \sigma_{n,q}}{\sigma_{k,q}} \intl_{O(n)} d\a  \intl_{V_{k,q}} d\u  \intl_{\pq}
F \left (\a \left[\begin{array} {c} \u\r^{1/2} \\ 0 \end{array} \right]  \right )\,  |\r|^{(k-q-1)/2}\, |\r|^{(n-k)/2} d\r.\nonumber\eea

Now we apply Lemma \ref{l2.3} again, but in the opposite direction, to obtain
\bea \label{chu}I&=&\frac{\sigma_{n,q}}{\sigma_{k,q}}  \intl_{O(n)} d\a \intl_{\frM_{k,q}}  F \left (\a \left[\begin{array} {c} \y \\ 0
\end{array} \right]  \right )\,  |\y|_q^{n-k} d\y\\
&=&\frac{\sigma_{n,q}}{\sigma_{k,q}}  \intl_{O(n)} d\a \,\Big ( \intl_{\bbr^k} \ldots  \intl_{\bbr^k} \Big )   F \left (\a \left[\begin{array} {c} \y \\ 0
\end{array} \right]  \right )\,  |\y|_q^{n-k}\, dy_1  \ldots dy_q \nonumber\\
\label{chu1}&=&\frac{\sigma_{n,q}}{\sigma_{k,q}} \intl_{G_{n,k}} d\xi \, \Big ( \intl_{\xi} \ldots  \intl_{\xi} \Big )  F(\x)\,|\x|_q^{n-k}\, d_\xi x_1  \ldots d_\xi x_q, \eea
which gives (\ref{BlP}).

\end{proof}

\begin{remark} For the case $q=k$, a certain analogue of (\ref{chu}) was obtained by  Forrester \cite[Proposition 4]{Fo} in different notation, including real, complex, and quaternionic cases. However, Forrester's result is somewhat incomplete because it does not contain  the transition from the  integration over square matrices to integration over subspaces that should be properly defined in the complex and quaternionic cases; cf. the transition from (\ref{chu}) to (\ref{chu1}). In the real case and $q=k$, this transition is performed by  Moghadasi \cite[p. 323]{Mog}.
\end {remark}

\section{Affine Blaschke-Petkantschin Formula and Drury's Identity}

The affine Blaschke-Petkantschin formula  (see Theorem \ref{BlP2vv} below) can be easily derived from  (\ref{BlP}). For the sake of completeness, we reproduce this derivation in Appendix, following the reasoning from Gardner \cite [Lemma 5.5]{Ga07} and removing unnecessary restrictions that were made in  \cite{Ga07}.  Once (\ref{BlP})  is established, the derivation of its affine version needs actually nothing but Fubini's Theorem.

We first observe that the volume $|\x|_q$ of the parallelepiped in (\ref{BlP}) can be replaced by the volume $\Del_q (\x)$ of the  convex hall of $\{0, x_1, \ldots, x_q \}$ by the known formula
\be\label{mkj}
\Del_q (\x)=\frac{1}{q!}\, |\x|_q. \ee
Given $\tilde \x=(x_0, x_1, \ldots, x_q) \in \frM_{n,q+1}$, let
  \be\label{mkj1} \Del_q(\tilde \x)\equiv \Del_q(x_0, x_1, \ldots, x_q)\ee
be the $q$-dimensional volume of the convex hall of $\{x_0, x_1, \ldots, x_q\}$. We denote by
   $\A_{n,k}$  the Grassmannian bundle of affine $k$-planes $\t$ with the standard measure $d\t$; see, e.g., \cite{Ru04d}.

\begin{theorem} \label{BlP2vv} Let $1\le q\le k\le n$. If $F \in L^1 (\frM_{n,q+1})$, then
\be\label{BlP2}
\intl_{\frM_{n,q+1}} \!  \!F(\tilde \x)\, d\tilde \x= c \intl_{\A_{n,k}}  d\t\intl_{\t^{q+1}} F(\tilde \x)\, \Del_q^{n-k}(\tilde \x)\, d_\t \tilde \x,\ee
where $\t^{q+1}$ is the collection of $q+1$ copies of $\t$,
\be\label{mkj10}c=\frac{(q!)^{n-k}\sigma_{n,q}}{\sigma_{k,q}},\ee
$d_\t  \tilde \x = d_\t x_0 \ldots  d_\t x_q $
 stands for the usual Lebesgue integration over $\t^{q+1}$.
\end{theorem}

Changing notation for the function $F$ in (\ref{BlP}) and (\ref{BlP2}), we obtain the following corollary.

\begin{corollary} \label{jik} Let $1\le q\le k\le n$. If $ F(\x) |\x|_q^{k-n} \in L^1 (\frM_{n,q})$, then
\be\label{BlPa}
\intl_{G_{n,k}} d\xi \intl_{\xi^q}F(\x)\, d_\xi \x  = \frac{\sigma_{k,q}}{\sigma_{n,q}} \intl_{\frM_{n,q}}  F(\x) |\x|_q^{k-n} \, d\x.  \ee

If $F(\tilde \x)\, \Del_q^{k-n}(\tilde \x)\in L^1 (\frM_{n,q+1})$, then
\be\label{BlP2b}
 \intl_{\A_{n,k}}  d\t \intl_{\t^{q+1}} F(\tilde \x)\, d_\t \tilde \x =\frac {(q!)^{k-n}\sigma_{k,q}} {\sigma_{n,q}} \intl_{\frM_{n,q+1}} \! F(\tilde \x)\, \Del_q^{k-n}(\tilde \x)\, d\tilde \x.\ee
\end{corollary}

An alternative proof of (\ref{BlPa}) and (\ref{BlP2b}) for nonnegative $F$, but without explicit constant and the explicit assumption for $F$, was given in  \cite[Section 5]{BL}.

 The right-hand sides of  (\ref{BlPa}) and (\ref{BlP2b}) can be treated as particular cases of the corresponding Riesz type distributions  on matrix spaces. For example, the right-hand side of  (\ref{BlPa}) agrees with the  Riesz distribution
 \be\label{UY} \z_F (\a) = \intl_{\frM_{n,q}}  F(\x) |\x|_q^{\a-n} \, d\x.  \ee
If  $F$ is smooth and rapidly decreasing as a function on $\bbr^{nq}$, the integral (\ref{UY}) is absolutely convergent provided   $Re \, \a> q-1$ and extends meromorphically  to all complex $\a$; see \cite[Lemma 4.2]{Ru06a}   for details.

Important additional features  of (\ref{BlPa}) and (\ref{BlP2b}) can be revealed if  we choose $F(\tilde \x)=f_1(x_0)....f_q(x_q)$ and write the right-hand sides as the corresponding multilinear forms. Then (\ref{BlP2b}) can be written in terms of the $k$-plane  transforms  of $f_j$ defined by $R_k f_j=\int_\t f_j$ with integration against the usual Lebesgue measure on $\t$. Specifically, for $1\le q\le k\le n$,
 \bea
 \intl_{\A_{n,k}}  \Big [\prod\limits_{j=0}^{q} (R_k f_j)(\t)\Big] d\t&=&\frac {(q!)^{k-n}\sigma_{k,q}} {\sigma_{n,q}}  \,
 \Big (\intl_{\rn}\cdots \intl_{\rn} \Big ) \Del^{k-n}_q(x_0, x_1, \ldots, x_q)\nonumber\\
\label{BlP2br} &\times& \left \{\prod\limits_{j=0}^{q}  f_j(x_j) \right \}
 dx_0\ldots dx_q.\eea
This equality agrees with the formula (2.11) in   Baernstein II and  Loss \cite{BL}. If $q=k$, it agrees with  Lemma 1 of Drury \cite{Dru84}. However, in both works, the constant on the right-hand side and conditions for $f_j$ are not specified.

Another consequence of (\ref{BlP2b}) can be obtained if we set $q=k<n$ and choose $F(\tilde \x)\equiv F(x_0,  \ldots, x_k)$ in the form
\be\label{Buu}
F(x_0,  \ldots, x_k)= \Big [\prod\limits_{j=0}^{k} f(x_j)\Big]\, \Big [\prod\limits_{j=k+1}^{k+\ell} \intl_{\t (x_0,  \ldots, x_k)} f(y_j)\, d_\t y_j\Big],\ee
where $\t (x_0,  \ldots, x_k)$ is a  $k$-plane containing the points $x_0, x_1, \ldots, x_k$, $f$ is a  function on $\rn$, and $\ell$ is a positive integer. The $k$-plane $\t (x_0, \ldots, x_k)$  is unique if $x_0, x_1, \ldots, x_k$ are in  general position. For every fixed $\t\in \A_{n,k}$ and $F$ of the form (\ref{Buu}) we have
\bea
\intl_{\t^{k+1}} F(\tilde \x)\, d_\t \tilde \x &=& \Big [\prod\limits_{j=0}^{k} \intl_\t f(x_j)\, d_\t x_j\Big]\,  \Big [\prod\limits_{j=k+1}^{k+\ell} \intl_{\t } f(y_j)\, d_\t y_j\Big]\nonumber\\
&=&  [(R_k f)(\t)]^{k+\ell+1}. \nonumber\eea
Hence (\ref{BlP2b}) yields
\bea
\label{Buu1} && \intl_{\A_{n,k}}   [(R_k f)(\t)]^{k+\ell+1} d\t= \frac {(k!)^{k-n}\sigma_{k,k}} {\sigma_{n,k}}\,
 \Big (\intl_{\rn} f(x_0)  \ldots \intl_{\rn} f(x_k) \Big )\qquad \quad \\
 &&\times  \Big [\prod\limits_{j=k+1}^{k+\ell} \intl_{\t (x_0,  \ldots, x_k)} f(y_j)\, d_\t y_j\Big]\,
\Del^{k-n}_k(x_0,  \ldots, x_k)\, dx_0, \ldots, dx_k.\nonumber\eea

This formula is well-justified provided that the right-hand side of it exists in the Lebesgue sense; cf. the  condition $F(\tilde \x)\, \Del_q^{k-n}(\tilde \x)\in L^1 (\frM_{n,q+1})$ for   (\ref{BlP2b}).

The case $k+\ell+1=n$ in (\ref{Buu1}) is known as {\it Drury's identity}; cf. formula (4) in \cite{Dru84}, which is understood in the sense of analytic continuation  according to Gelbart  \cite{Ge}.

\noindent{\bf Concluding Remark.}
An analogue of the polar decomposition in Lemma \ref{l2.3} is known for complex and quaternionic matrices and in the context of formally real Jordan
algebras; see, e.g., Zhang \cite[Lemma 3.1]{Zha1},   Moghadasi \cite[Theorem 2.5]{Mog},  Forrester \cite{Fo},  Faraut and  Travaglini  \cite[Section 4]{FT}. We believe that our proof of Theorem \ref{the} and  its consequences extend to the corresponding more general settings.

\section{Appendix: Proof of Theorem \ref{BlP2vv}}

Fix any $x_0 \in \rn$  and set $\tilde \x =(x_0, \x)= (x_0, x_1, \ldots, x_q) \in \frM_{n,q+1}$. We  replace   $F(\x)\equiv F (x_1, \ldots, x_q)$ by $F (x_0, x_1 + x_0 , \ldots, x_q + x_0)$ in (\ref{BlP}) to get
\bea
&&\intl_{\frM_{n,q}} F(x_0,\x)\, d\x =\intl_{\frM_{n,q}} F (x_0, x_1 + x_0 , \ldots, x_q + x_0)\, d\x\nonumber\\
&&=\frac{\sigma_{n,q}}{\sigma_{k,q}} \intl_{G_{n,k}} d\xi \intl_{\xi^q} F (x_0, x_1 + x_0 , \ldots, x_q + x_0)\, |\x|_q^{n-k}(\x)\, d_\xi \x.\nonumber\eea
Note that
\[ \frac{1}{q!}\, |\x|_q=\Del_q (\x)\equiv \Del_q  (0, x_1, \ldots, x_q)=\Del_q (x_0, x_1 + x_0 , \ldots, x_q +x_0);\]
 see (\ref{mkj}) and (\ref{mkj1}). Hence
\bea
\label{lmy} \intl_{\frM_{n,q}} F(x_0,\x)\, d\x &=&c\intl_{G_{n,k}} d\xi \intl_{\xi^q} F (x_0, x_1 + x_0 , \ldots, x_q + x_0)\nonumber\\
\label{lmy} &\times&[\Del_q (x_0, x_1 + x_0 , \ldots, x_q +x_0)]^{n-k}\, d_\xi \x,\eea
where
\[c=\frac{(q!)^{n-k}\sigma_{n,q}}{\sigma_{k,q}}.\]
Now we integrate (\ref{lmy}) in the $x_0$-variable and change the order of integration. This gives
\bea
 \intl_{\frM_{n,q+1}} F(\tilde \x)\, d\tilde \x&=&c\intl_{G_{n,k}} d\xi \intl_{\rn} dx_0 \intl_{\xi^q}  F (x_0, x_1 + x_0 , \ldots, x_q + x_0)\nonumber\\
  &\times& [\Del_q (x_0, x_1 + x_0 , \ldots, x_q +x_0)]^{n-k} \, d_\xi x_1\cdots d_\xi x_q.\nonumber\eea
Let $x_0= y_0 +z_0$, where $y_0 \in \xi$ and  $z_0 \in \xi^\perp$, so that  $\t=\xi +z_0$ is an affine $k$-plane. Then the right-hand side becomes
\bea
 &&c\intl_{G_{n,k}} d\xi \intl_{\xi^\perp} dz_0\intl_{\xi} dy_0\intl_{\xi^q}  F (y_0 +z_0, x_1 + y_0 +z_0 , \ldots, x_q + y_0 +z_0)\nonumber\\ &&\times [\Del_q (y_0 +z_0, x_1 + y_0 +z_0 , \ldots, x_q + y_0 +z_0)]^{n-k} d_\xi x_1\cdots d_\xi x_q.\nonumber\eea
Setting $y_0 =\eta_0, \;   x_1 + y_0=\eta_1, \ldots,  x_q + y_0=\eta_q$, we write this expression as
\bea
&&c\intl_{G_{n,k}} \! d\xi \intl_{\xi^\perp} dz_0\intl_{\xi} d\eta_0\intl_{\xi^q}  F (\eta_0 +z_0, \eta_1 +z_0 , \ldots, \eta_q +z_0)\nonumber\\ &&\times[\Del_q (\eta_0 +z_0, \eta_1 +z_0 , \ldots, \eta_q +z_0)]^{n-k} d_\xi \eta_1\cdots d_\xi \eta_q\nonumber\\
&&=c\intl_{\A_{n,k}}\! d\t\intl_{\t^{q+1}} \! \!  F (x_0, x_1, \ldots, x_q)\,  [\Del_q (x_0, x_1, \ldots, x_q)]^{n-k}\, d_\t x_0\cdots d_\t x_q, \nonumber\eea
which gives (\ref {BlP2}). ${}$ \hfill $\Box$

\bibliographystyle{amsplain}

\end{document}